\documentclass[10pt,oneside,a4paper,reqno]{amsart}
\RequirePackage{fix-cm} 
\usepackage{lmodern}
\usepackage[T1]{fontenc}
\usepackage[english]{babel}
\usepackage{amssymb,amsmath,amsthm,graphicx,amsfonts,mathtools}
\usepackage{ulem,lipsum}

\usepackage[dvipsnames]{xcolor}

\usepackage[
  rm={oldstyle=false,proportional=true},
  sf={oldstyle=true,proportional=true},
  tt={oldstyle=true,proportional=true,variable=true},
  qt=false,
]{cfr-lm}
\usepackage{cite,enumerate,setspace}
\usepackage{hyperref,cleveref}
\hypersetup{colorlinks,
citecolor=RoyalBlue,
linkcolor=RoyalBlue,
urlcolor=RoyalBlue}

\usepackage{geometry}
\makeatletter
\def\@seccntformat#1{\hspace*{0mm}%
 \protect\textup{\protect\@secnumfont
   \ifnum\pdfstrcmp{subsection}{#1}=0 \bfseries\fi
   \csname the#1\endcsname
   \protect\@secnumpunct
     }%
}

\newtheorem{lemma}{Lemma}
\newtheorem{proposition}{\sbweight Proposition}

\newtheorem{theorem}{\sbweight Theorem}
\theoremstyle{remark}
\newtheorem{remark}{\sc  Remark\rm}[section]

\catcode`\>=\active \def>{
\fontencoding{T1}\selectfont\symbol{62}\fontencoding{\encodingdefault}}
\newcommand{\assign}{\coloneqq}

\newcommand{\mathd}{\mathrm{d}}
\newcommand{\tmabbr}[1]{#1}

\newcommand{\tmem}[1]{{\em #1\/}}
\newcommand{\tmmathbf}[1]{\ensuremath{\boldsymbol{#1}}}

\newcommand{\tmname}[1]{\textsc{#1}}
\newcommand{\tmop}[1]{\ensuremath{\operatorname{#1}}}

\newcommand{\tmtextit}[1]{{\itshape{#1}}}

\newcommand{\tmtextsc}[1]{{\scshape{#1}}}

\newcommand{\argminE}{\mathop{\mathrm{argmin}}}

\newcommand{\Ekg}{\mathcal{E}_{\kappa, \gamma}}

\newcommand{\tmu}{\tmmathbf{u}}
\newcommand{\tmv}{\tmmathbf{v}}
\newcommand{\w}{\tmmathbf{w}}
\newcommand{\Othreethree}{O (3, \tmmathbf{e}_3)}
\newcommand{\divv}{div}
\newcommand{\Fsp}{\hspace{0.12em} : \hspace{0.09em}}
\newcommand{\st}{\, :\,}

\newcommand{\abs}[1]{\lvert #1 \rvert}

\newcommand{\e}{\tmmathbf{e}}
\newcommand{\n}{\tmmathbf{n}}

\newcommand{\m}{\ensuremath{\tmmathbf{m}}}
\newcommand{\vv}{\tmmathbf{v}}
\newcommand{\ww}{\tmmathbf{w}}

\newcommand{\Om}{\ensuremath{\Omega}}

\newcommand{\grad}{\nabla}

\newcommand{\RR}{\mathbb{R}}
\newcommand{\NN}{\mathbb{N}}
\newcommand{\Stwo}{\mathbb{S}}

\newcommand{\eqs}{=}

\begin{document}

\title[Symmetry of minimizers]{
Symmetry properties of minimizers of a perturbed Dirichlet energy with a boundary penalization 
}

\author{Giovanni Di Fratta}

\author{Antonin Monteil}

\author{Valeriy Slastikov}


\begin{abstract}
We consider \(\mathbb{S}^2\)-valued maps on a domain \(\Omega\subset\mathbb{R}^N\) minimizing a perturbation of the Dirichlet energy with vertical penalization in \(\Omega\) and horizontal penalization on \(\partial\Omega\). We first show the global minimality of universal constant configurations in a specific range of the physical parameters using a Poincar{\'e}-type inequality. Then, we prove that any energy minimizer takes its values into a fixed meridian of the sphere \(\mathbb{S}^2\), and deduce uniqueness of minimizers up to the action of the appropriate symmetry group. We also prove a comparison principle for minimizers with different penalizations. Finally, we apply these results to a problem on a ball and show radial symmetry and monotonicity of minimizers. In dimension $N=2$ our results can be applied to the Oseen--Frank energy for nematic liquid crystals and micromagnetic energy in a thin-film regime.
\end{abstract}

\clearpage\maketitle
\thispagestyle{empty}

\section{Introduction}

The motivation for this study comes from the field of thin structures --- a branch of materials science that is currently 
experiencing rapid growth. The interest in thin structures relies on their applications in miniaturization and integration of electronic devices, but even more on their capability to support the emergence of new physics \cite{gibertini19, hellman17,sheka21}. Indeed, atomically thin materials can be employed to achieve physical properties that are hardly visible in bulk materials. Moreover, combining several atomically thin layers to create new heterostructures allows for the design of novel materials with prescribed properties \cite{novoselov16}.


In the last twenty years, the thin-structures in micromagnetics and nematic liquid crystals have been an area of active research in both applied mathematics and condensed matter physics (see, e.g., \cite{Carbou01,  Desim2002, Desim2006, DiF201,DiF202, Gioia97, Golo1, Golo2, Ignat, Nit2018, Slast05}). Recent advances in manufacturing thin films and curved layers provide a possibility to design new materials composed of several magnetic monolayers of atomic thickness \cite{gibertini19, streubel16}. These new materials exhibit some unconventional properties, including perpendicular magnetocrystalline anisotropy \cite{dieny17} and Dzyaloshinskii–Moriya interaction (DMI) (or antisymmetric exchange) \cite{dz58, moriya60} and require a new set of reduced theoretical models to predict the magnetization behavior in ferromagnetic samples.  This new physics is often dominated by surface and edge effects, and leads to a surprising behavior near the material boundaries, giving rise to novel magnetization structures \cite{hellman17, muratov17, sheka21, zheng18}.

In this paper, we are interested in studying the ground states of a simplified model (cf.~eq.~\eqref{eq:enpen}), concentrating on their symmetry properties. The model we investigate is closely related to a reduced model for ferromagnetic thin films with strong perpendicular anisotropy in the regime when magnetocrystalline and shape anisotropies have the magnitude of the same order, leading to the preference for in-plane magnetization inside the sample and out-of-plane magnetization behavior on the boundary \cite{kohn05, DMS21}.

Since the energy functionals governing micromagnetic interactions and defects in nematic liquid crystals are mathematically related, our analysis also applies to the analysis of ground states in the thin-film limit Oseen--Frank theory of nematic liquid crystals under weak anchoring conditions.

\subsection{Our model}
Even though our main motivation comes from the study of thin film structures, we formulate and prove our results for domains of arbitrary dimension. Let $\Om \subset \RR^N$, $N\in\mathbb{N}^\ast$, be a smooth bounded domain, and let $\Stwo^2 \subset \RR^3$ be the two-dimensional unit sphere. We consider the energy of a configuration $\m \in H^1 (\Om, \Stwo^2)$, defined by
\begin{equation}
  \mathcal{E}_{\kappa} (\m) = \int_{\Om} \abs{\nabla \m}^2 +
  \kappa^2  \int_{\Om} (\m \cdot \e_3)^2, 
\end{equation}
where $\abs{\nabla \m}^2 = \sum_{i = 1}^N \left| \partial_i \m \right|^2$, $\e_3$=(0,0,1), and $\kappa \in [0,+\infty)$ is some fixed (material-dependent) parameter which takes into account in-plane anisotropic effects. Under natural boundary conditions (which
is the typical case in micromagnetics), the only minimizers of
$\mathcal{E}_{\kappa}$ are the constant in-plane configurations. In this note, we are interested in the problem of minimizing the energy
$\mathcal{E}_{\kappa}$ under an additional penalization term on the boundary of $\Om$ that makes the problem non-trivial: for every $\gamma>0$ we consider the energy functional defined for every $\m \in H^1 (\Om, \Stwo^2)$ by
\begin{align}
  \Ekg (\m) = \int_{\Om} \abs{\nabla \m}^2 + \kappa^2  \int_{\Om} (\m \cdot
  \e_3)^2 + \; \frac{1}{\gamma^2}  \int_{\partial \Om} \abs{\m \times \e_3}^2,
  \label{eq:enpen}
\end{align}
where $\gamma \in (0, + \infty)$ fixes the intensity of the perpendicular anisotropy on $\partial \Om$. The energy in this form naturally appears in the Oseen-Frank model of liquid crystals \cite{frank58} and as a thin film limit of micromagnetic energy for ferromagnetic materials with strong perpendicular anisotropy \cite{DMS21}.

A straightforward application of the Direct methods of the Calculus of
Variations assures that for every $\kappa \in [0,+\infty)$ and $\gamma \in (0, +\infty)$, there exists a global minimizer of the energy $\Ekg$. The minimizers satisfy the following Euler-Lagrange equations in the weak sense, i.e.,  for every
$\tmmathbf{\varphi} \in H^1 \left( \Om, \RR^3 \right)$
\begin{align}
  \int_{\Om} \grad \m \st \grad \tmmathbf{\varphi}+ \kappa^2 \left( \m \cdot
  \tmmathbf{e}_3 \right) (\tmmathbf{\varphi} \cdot \tmmathbf{e}_3) & \eqs
  \int_{\Om} (| \grad \m |^2 + \kappa^2  \left( \m \cdot \tmmathbf{e}_3
  \right)^2) \m \cdot \tmmathbf{\varphi} \nonumber\\
  &   \qquad + \frac{1}{\gamma^2} \int_{\partial \Om} \left[ \left( \m \cdot
  \tmmathbf{e}_3 \right) \e_3 - \left( \m \cdot \tmmathbf{e}_3 \right)^2 \m
  \right] \cdot \tmmathbf{\varphi}. \label{eq:ELweak} 
\end{align}
If a global minimizer \(\m\) is \(\mathcal{C}^1(\overline\Omega,\Stwo^2) \cap \mathcal{C}^2(\Omega,\Stwo^2)\), this means that \(\m\) classically solves 
\begin{equation}
    - \Delta \m + \kappa^2 \left( \m \cdot \tmmathbf{e}_3 \right)
    \tmmathbf{e}_3 \eqs (| \nabla \m |^2 + \kappa^2  \left( \m \cdot
    \tmmathbf{e}_3 \right)^2) \m 
    \quad\text{in \Om}
    \label{eq:ELstrong}
\end{equation}
 together with the nonlinear Robin boundary condition:
\begin{equation}
    \partial_{\n} \m \eqs \frac{1}{\gamma^2} \left[ \left( \m \cdot
    \tmmathbf{e}_3 \right) \e_3 - \left( \m \cdot \tmmathbf{e}_3 \right)^2 \m
    \right]
    \quad \text{on \(\partial \Om\).}
\end{equation}

In the limiting case \(\gamma\to 0\), we have a non-trivial Dirichlet boundary value problem since $\Ekg$ tends to the energy $\mathcal{E}_{\kappa, 0}$ defined for every $\m \in H^1 (\Om,\Stwo^2)$ as
\begin{equation}
  \mathcal{E}_{\kappa, 0} (\m) \assign \left\{\arraycolsep=1.4pt\def\arraystretch{1.6}\begin{array}{ll}
    \mathcal{E}_{\kappa} (\m) & \text{if $\m\pm \e_3\in H^1_0(\Om,\RR^3)$},\\
    + \infty & \tmop{otherwise} .
  \end{array}\right. \label{eq:enstrongpen}
\end{equation}

Note that, we shall write $\Ekg$ for both the boundary penalization problem, corresponding to {\eqref{eq:enpen}} when \(\gamma>0\), and the boundary value problem (with boundary value \(\pm \tmmathbf{e}_3\)), corresponding to \eqref{eq:enstrongpen} when \(\gamma=0\). This is more convenient since many of our results apply to both problems. As in the case \(\gamma>0\), the existence of global minimizers for \(\mathcal{E}_{\kappa,0}\) follows from the Direct methods in the Calculus of Variations.

\subsection{Contributions of the present work}
The aim of the paper is to show the symmetry and uniqueness properties of minimizers of $\Ekg$. In particular, we prove that any minimizer of $\Ekg$ has values in some meridian of the sphere and is unique up to the symmetries in the group of isometries preserving the $\e_3$-axis. 
As a consequence, restricting domain $\Om$ to a ball we also show that any minimizer is radially symmetric and monotone. 

In what follows, we describe the results in more detail.
Our first result concerns the minimality of universal configurations, i.e.,  vector fields $\m
\in H^1 \left( \Om, \Stwo^2 \right)$ which solve the Euler-Lagrange equations
{\eqref{eq:ELweak}} regardless of the value of the boundary
penalization constant $\gamma>0$. Given the dependence  of the boundary term in {\eqref{eq:ELweak}} on $\gamma$, such configurations must satisfy
\begin{equation}
( \m \cdot \tmmathbf{e}_3 ) (\tmmathbf{e}_3-(\m\cdot \tmmathbf{e}_3)\m ) = 0 \quad \text{{\tmabbr{a.e\ }}on
  } \partial \Om .
\end{equation}
It is easy to check that the {\it{constant}} vector fields $\pm
\tmmathbf{e}_3$, as well as any {\it{constant}}  in-plane vector field
$\tmmathbf{e}_{\bot} \in \Stwo^2$, $\e_{\perp} \cdot \e_3 = 0$, are universal
configurations. Concerning these configurations, we prove the following
result, which clarifies how to tune the parameters $\kappa$ and $\gamma$ so
that these configurations emerge as ground states.

\begin{theorem}
  \label{prop:univconf}Let $N\in\mathbb{N}^\ast$ and $\Om \subseteq \RR^N$ be a smooth bounded domain. The following assertions hold:
  \begin{enumerate}[i)]
    \item\label{universal_outofplane} For any $\gamma\in [0,+\infty)$, there exists $\kappa_{\gamma} > 0$, depending only on $\gamma$ and $\Om$, such that for any $\kappa\in [0,\kappa_{\gamma})$ the constant out-of-plane vector fields $\pm \e_3$ are the unique global minimizers of $\Ekg$. In particular, $\pm \e_3$ are the unique solutions of the Dirichlet boundary value problem $\min\mathcal{E}_{\kappa,0}$ if $\kappa\in [0,\kappa_0)$.
     \item\label{universal_inplane} For any $\kappa \in (0,+\infty)$, there exists $\gamma_{\kappa} > 0$, depending only on $\kappa$ and {\Om}, such that for any $\gamma\in (\gamma_{\kappa},+\infty)$ the constant in-plane vector fields $\tmmathbf{e}_{\bot} \in \Stwo^2$, $\e_{\perp} \cdot \e_3 = 0$, are the only global minimizers of $\Ekg$.
  \end{enumerate}
\end{theorem}

The statements in \Cref{prop:univconf} characterize the energy
landscape under restrictions on the control parameters
$\kappa$ and $\gamma$. Our second result retrieves information on the properties of 
minimizers under no additional assumptions on the system parameters
$\kappa$ and $\gamma$. Exploiting the symmetries of the system, we prove that
minimizers of $\Ekg$ are smooth up to the boundary of $\Om$ and have values in a quadrant of a meridian in \(\Stwo^2\).

In what follows, we denote by
\[ \Othreethree \assign \left\{ \sigma \in O (3) \st \sigma \left( \e_3
   \right) = \e_3 \text{ or } \sigma \left( \e_3 \right) = - \e_3 \right\} \]
the group of isometries preserving the $\e_3$-axis.

\begin{theorem}
  \label{structureMinimizers} Let $N\in\mathbb{N}^\ast$ and $\Om \subset \RR^N$  be a smooth bounded domain and let $\kappa,\gamma \in [0, + \infty)$. If $\m \in H^1 (\Om, \Stwo^2)$ is a global minimizer of $\Ekg$, then \(\m\in \mathcal C^\infty(\overline\Om, \Stwo^2)\), and there exist $\sigma \in \Othreethree$ and a lifting map $\varphi\in \mathcal C^\infty(\overline\Om)$ such that \(0\leqslant\varphi\leqslant\frac \pi 2\)  and 
    \[ \m = \sigma \circ (\sin \varphi, 0, \cos \varphi) \quad \text{\ in }
       \Om . \]
Moreover, when $\gamma>0$, we have
\begin{align}
 \label{phasePenalization} \varphi\in\argminE_{\psi\in H^1(\Omega)}\left\{ \int_{\Om}
    \abs{\nabla \psi}^2 + \kappa^2  \int_{\Om} \cos^2 \psi +
    \frac{1}{\gamma^2}  \int_{\partial \Om} \sin^2 \psi\right\},
\end{align}
and either \(\varphi\equiv 0\) in \(\overline \Om\) so that \(\m\equiv \pm\e_3\), or \(\varphi\equiv\frac\pi 2\) in \(\overline \Om\) so that \(\m\) is constant in-plane (i.e., $\m \cdot \e_3\equiv 0$), or \(0<\varphi<\frac\pi 2\) \ in $\overline \Om$.

In the case $\gamma=0$ we have
 \begin{align}
     \label{phaseDirichletProblem} \varphi\in\argminE_{\psi\in H_0^1(\Omega)}\left\{ \int_{\Om}
    \abs{\nabla \psi}^2 + \kappa^2  \int_{\Om} \cos^2 \psi\right\},
\end{align}
and either \(\varphi\equiv 0\) in \(\overline \Om\) so that \(\m\equiv \pm\e_3\) or \(0<\varphi<\frac\pi 2\) \ in $ \Om$.
%
\end{theorem}
When $\gamma=0$, the minimization problem \eqref{phaseDirichletProblem} has a unique solution such that \(\operatorname{Im}(\varphi)\subset (0,\frac\pi 2]\). This follows by classical results about sublinear elliptic equations  which immediately apply to the Dirichlet problem \eqref{phaseDirichletProblem} because the map \(t\mapsto (\sin t)/t\) is decreasing in \((0,\pi]\) (see Appendix II in \cite{brezisSublinearEllipticEquations1992}). Note, however, that the argument does not assure uniqueness of solutions because, in principle, there is the possibility of having coexistence of a solution \(\varphi\) such that \(\operatorname{Im}(\varphi)\subset (0,\frac\pi 2]\) with the constant solution \(\varphi\equiv 0\) (corresponding to \(\m\equiv \pm \e_3\)). More importantly, these classical results do not to fit with the case $\gamma >0$, i.e., do not apply to minimizers of \eqref{phasePenalization}. Our next result resolves these issues as it assures uniqueness of minimizers of \eqref{phaseDirichletProblem} and \eqref{phasePenalization} up to isometries in \(\Othreethree\).

\begin{theorem}\label{corollaryDirichlet}
Let $N\in\mathbb{N}^\ast$, let $\Om \subset \RR^N$ be a smooth bounded domain, and let \(\kappa\in[0,+\infty)\), $\gamma\in[0,+\infty)$. If \(\m\) and \(\Bar\m\) are two minimizers of the energy $\mathcal{E}_{\kappa,\gamma}$, then there exists \(\sigma\in\Othreethree\) such that \(\Bar\m=\sigma\circ\m\).
\end{theorem}

\begin{figure}[t]
  \scalebox{0.8}{\includegraphics{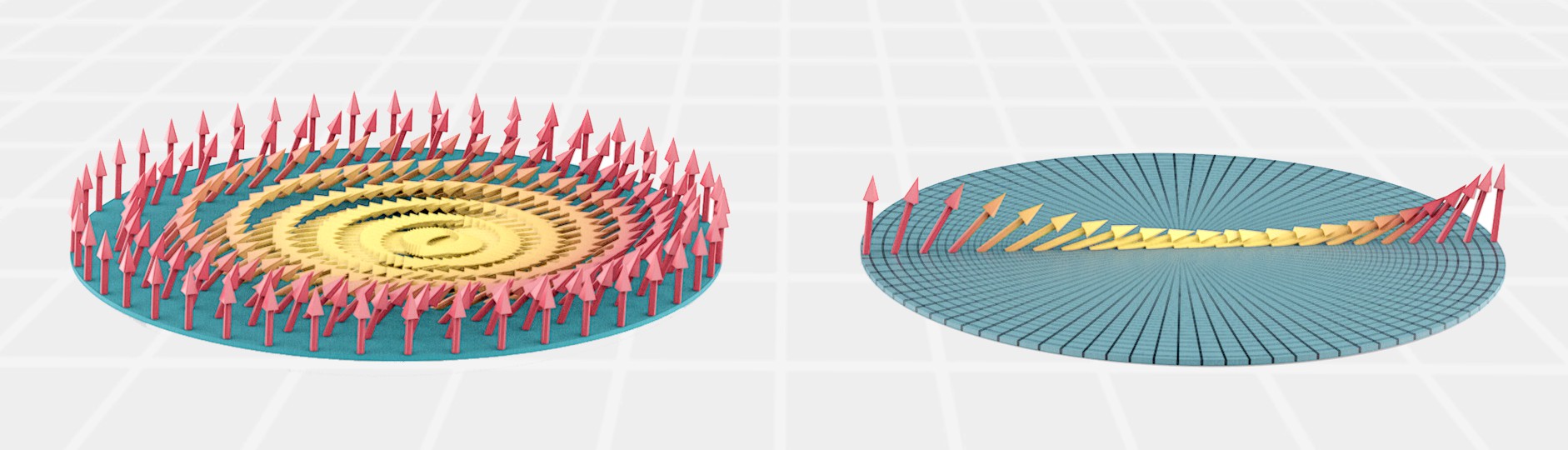}}
  \caption{\label{fig:1}On the left, a minimizer of $\Ekg$, with $\kappa^2 =
  5$, $\gamma = 0.1$, in the unit disk of $\RR^2$. On the right, we isolated a ray in order to visualize the profile of the minimizer better.
  }
\end{figure}

We also have a comparison principle for solutions with different values of \(\kappa,\gamma\):
\begin{theorem}
\label{theoremComparison}
Let $N\in\mathbb{N}^\ast$, \(\Om\subset\RR^N\) be a smooth bounded domain, and \(\kappa_1,\kappa_2,\gamma_1,\gamma_2\in [0,+\infty)\) with \(\kappa_1\leqslant\kappa_2\), \(\gamma_1\leqslant\gamma_2\) and \((\kappa_1,\gamma_1)\neq (\kappa_2,\gamma_2)\). If for \(i=1,2\), \(\varphi_i \) is a solution of \eqref{phasePenalization} valued into \([0,\frac\pi 2]\), with \((\kappa,\gamma)=(\kappa_i,\gamma_i)\), then we have either \(\varphi_1=\varphi_2\equiv 0\), or \(\varphi_1=\varphi_2\equiv \frac\pi 2\), or \(\varphi_1<\varphi_2\) in \(\Omega\). If in addition we assume \(\gamma_2>0\), then in the latter case we actually have  \(\varphi_1<\varphi_2\) on \(\overline\Omega\).
\end{theorem}
The statements in \Cref{prop:univconf,structureMinimizers,corollaryDirichlet,theoremComparison} hold without any assumption on the
geometry of the domain $\Om$. Our last result focuses on the case where $\Om$ is a ball, and shows that any global minimizer of $\Ekg$ is radially symmetric in this case, i.e., $\m = \m (|x|)$, and by Theorem~\ref{structureMinimizers} has values in a quadrant of a meridian (an example of a minimizer of \eqref{phasePenalization} on the two-dimensional ball is illustrated in \Cref{fig:1}).

\begin{theorem}
  \label{prop:symm}Let $\kappa,\gamma \in [0, + \infty)$. Let $\Om=B_R$ be a ball of radius \(R>0\) centered at the origin in \(\RR^N\), then any global minimizer \(\m\) of $\Ekg$ is radially symmetric. More precisely, there exists $\sigma \in \Othreethree$ such that
  \[ \sigma\circ\m (x) =  \Bigl(\sin \Bigl(\frac{u(\abs{x})}{2}\Bigr),0, \cos \Bigl(\frac{u(\abs{x})}{2}\Bigr)\Bigr) \quad \text{in } \Om 
\]
for some non-increasing \(\mathcal{C}^\infty\) function \(u:[0,R]\to[0,\pi]\) which solves the nonlinear ODE
\begin{align}\label{equationELradial}
&u''(r)+\frac{N-1}{r}\, u'+\kappa^2\sin u=0\quad\text{in \((0,R)\)},
\\
\label{boundaryConditionELleft}
&u'(0)=0,
\end{align} 
with either a Dirichlet condition or a nonlinear Robin boundary condition at \(r=R\), namely
\begin{equation}\label{boundaryConditionEL}
\begin{cases}
u(R)=0&\text{if \(\gamma=0\),}\\
u'(R)+\frac{1}{\gamma^2}\sin u(R)=0&\text{if \(\gamma>0\).}
\end{cases}
\end{equation}
\end{theorem}
By \Cref{corollaryDirichlet}, the function \(u\) in \Cref{prop:symm} is unique when \(\gamma \in [0,+\infty)\). It is either the steady state \(u\in  \{ 0, \pi\}\) or a decreasing function into \((0,\pi)\).


\subsection{Outline}The paper is organized as follows. In Section~\ref{sec:2},
we prove the minimality of universal configurations in a specific range of the parameters \(\kappa,\gamma\)
(\Cref{prop:univconf}). For that, we need a Poincar{\'e}-type inequality
with a boundary term, which is proved in \Cref{prop:PIR}.
Section~\ref{sec:3} is devoted to the analysis of symmetries of the minimizers
and their range, there we prove \Cref{structureMinimizers}. In
Section~\ref{sec:4}, we show the uniqueness of minimizers of $\Ekg$ for $\gamma \in [0, +\infty)$ up to isometries in \(\Othreethree\), see \Cref{corollaryDirichlet}. In Section~\ref{sec:compsols} we prove our comparison result (Theorem~\ref{theoremComparison}) which states that solutions of \eqref{phasePenalization} are order preserving in $\kappa$ and $\gamma$. Finally, in
Section~\ref{sec:5}, we focus on the case when the domain is a ball, and
we prove radial symmetry and monotonicity of solutions of \eqref{phasePenalization} (\Cref{prop:symm}).

\section{Minimality of universal configurations:
\texorpdfstring{Proof of \Cref{prop:univconf}}{}}\label{sec:2}

To investigate the minimality of the constant out-of-plane configurations $\pm
\tmmathbf{e}_3$ we need the following Poincar{\'e}-type inequality, which can
be of some interest on its own.

\begin{lemma}[Poincar{\'e}-type inequality]
  \label{prop:PIR}Let $\Om \subseteq \RR^N$ be a bounded smooth domain. Then, there exists $c_{\Om} > 0$ such that for every $u \in H^1 \left(\Om \right)$ and every $\delta > 0$, we have
  \begin{equation}
    \delta \left( c_{\Om} - \delta \right) \int_{\Om} u^2 (x) \mathd x \;
    \leqslant \int_{\Om} \left| \grad u (x) \right|^2 \mathd x + \delta
    \int_{\partial \Om} u^2 (x) \mathd\mathcal{H}^{N-1}(x). \label{eq:PINB}
  \end{equation}
  Moreover, in the previous relation, the constant $c_{\Om}$ can be taken $c_{\Om}= \frac{N}{\tmop{diam} \left( \Om \right)}$.
\end{lemma}

\begin{proof}
  We argue along the lines in \cite{Dif2021}. Without loss of generality, we can assume that $0 \in \Om$. Also, by
  density, it is sufficient to prove {\eqref{eq:PINB}} for every $u \in
  \mathcal C^{\infty} \left( \overline{\Om} \right)$. By the divergence theorem we have
  \[ \int_{\Om} \left( 2 u (x) \grad u (x) \cdot x + N u^2 (x) \right) \mathd
     x \; \eqs \; \int_{\Om} \mathrm{{\divv}} [u^2 (x) x] \mathd x \; \eqs \;
     \int_{\partial \Om} u^2 (x) \n(x)\cdot x\, \mathd\mathcal{H}^{N-1}( x), \]
   where, for a.e.\ $x \in \partial \Om$, we denoted by $\n (x)$ the outward-point unit normal at $x \in \partial \Om$. By Young's inequality, it follows that for every $\delta > 0$ one has
  \begin{equation}
    N \int_{\Om} u^2 (x) \mathd x  \leqslant \sup_{x \in \partial \Om} | x
    |_{} \int_{\partial \Om} u^2 (x) \mathd\mathcal{H}^{N-1}(x) + \sup_{x \in \Om} | x |_{}
    \int_{\Om} \left[ \frac{1}{\delta} \left| \grad u (x) \right|^2 + \delta
    u^2 (x) \right] \mathd x. \nonumber
  \end{equation}
  Since $\sup_{x \in \partial \Om} | x | \leqslant \mathrm{\tmop{diam}} \left(
  \Om \right)$ and $\sup_{x \in \Om} | x | \leqslant \mathrm{\tmop{diam}} \left(
  \Om \right)$, we have
  \begin{equation}
    \left( N - \delta \mathrm{\, \tmop{diam}} \left( \Om \right) \right)
    \int_{\Om} u^2 (x) \mathd x  \leqslant 
    \frac{\mathrm{\mathrm{\tmop{diam}} \left( \Om \right)}}{\delta} \int_{\Om}
    \left| \grad u (x) \right|^2 \mathd x + \mathrm{\mathrm{\tmop{diam}}
    \left( \Om \right)} \int_{\partial \Om} u^2 (x) \mathd\mathcal{H}^{N-1}(x) . \nonumber
  \end{equation}
  From the previous estimate, we get that for every $\delta > 0$ there holds
  \[ \delta \left( c_{\Om} - \delta \right) \int_{\Om} u^2 (x) \mathd x \;
     \leqslant \int_{\Om} \left| \grad u (x) \right|^2 \mathd x + \delta
     \int_{\partial \Om} u^2 (x) \mathd\mathcal{H}^{N-1}(x), \]
  with $c_{\Om} \assign \frac{N}{\tmop{diam} \left( \Om \right)}$. This
  concludes the proof.
\end{proof}

\begin{proof}[Proof of \Cref{prop:univconf}, \cref{universal_outofplane}] We first consider the case where \(\gamma>0\). Without loss of generality, we can focus on the configuration $\m = + \e_3$.
  We observe that for any $\vv \in H^1 (\Om, \RR^3)$ such that $| \vv
  +\tmmathbf{e}_3 | = 1$ or, equivalently, such that $\left| \vv \right|^2 = -
  2 \left( \vv \cdot \tmmathbf{e}_3 \right)$, we have
  \begin{align}
    \Ekg  (\tmmathbf{e}_3 + \vv) - \Ekg (\tmmathbf{e}_3) & \eqs  \int_{\Om}
    \left| \grad \vv \right|^2 + \kappa^2  \int_{\Om} \left( \vv \cdot
    \tmmathbf{e}_3 \right)^2 + 2 \left( \vv \cdot \tmmathbf{e}_3 \right) +
    \frac{1}{\gamma^2} \int_{\partial \Om} \abs{ \vv \times \tmmathbf{e}_3 }^2 \nonumber\\
    & \eqs  \int_{\Om} \left| \grad \vv \right|^2 - \kappa^2  \int_{\Om} |
    \tmmathbf{v}_{\bot} |^2 \; + \; \frac{1}{\gamma^2} \int_{\partial \Om} |
    \tmmathbf{v}_{\bot} |^2,  \label{eq:EnInc}
  \end{align}
  with $\vv_{\bot} =\tmmathbf{v}- \left( \tmmathbf{v} \cdot \e_3 \right)
  \e_3$. Estimating the energy increment $\Ekg  (\tmmathbf{e}_3 + \vv) - \Ekg
  (\tmmathbf{e}_3)$ through the Poincar{\'e} inequality {\eqref{eq:PINB}} we get for every \(\delta>0\),
  \begin{align}
    \Ekg (\tmmathbf{e}_3 + \vv) - \Ekg (\tmmathbf{e}_3) & \geqslant \delta
    \left( c_{\Om} - \delta \right) \int_{\Om} | \tmmathbf{v}_{\bot} |^2 -
    \delta \int_{\partial \Om} | \tmmathbf{v}_{\bot} |^2 - \kappa^2 
    \int_{\Om} | \tmmathbf{v}_{\bot} |^2 \; + \; \frac{1}{\gamma^2}
    \int_{\partial \Om} | \tmmathbf{v}_{\bot} |^2 \nonumber\\
    & \geqslant  \left( \delta \, \left( c_{\Om} - \delta \right) - \kappa^2
    \right) \int_{\Om} | \tmmathbf{v}_{\bot} |^2 + \left( \;
    \frac{1}{\gamma^2} - \delta \right) \int_{\partial \Om} |
    \tmmathbf{v}_{\bot} |^2 . \label{optimalityOutOfPlane}
  \end{align}
If we set \(\delta_\gamma\assign\min\{\frac{c_\Om}{2},\frac{1}{\gamma^2}\}\) and \(\kappa_{\gamma}\assign (\delta_\gamma(c_\Omega-\delta_\gamma))^{1/2}>0\), then for every \(\kappa\in [0,\kappa_\gamma)\) there exists $\delta\in (0,\delta_\gamma)$ such that \(\delta \, \left( c_{\Om} - \delta \right)>\kappa^2\) and \(\frac{1}{\gamma^2}>\delta\). Hence, by \eqref{optimalityOutOfPlane}, $\e_3$ (and so \(-\e_3\)) is a minimum point of $\Ekg$, and any other minimum point \(\m\) can only be obtained by perturbations in the \(\e_3\) direction. This means that the constant out-of-plane vector fields \(\pm\e_3\) are the only minimizers of \(\Ekg\). 
  
A simpler argument gives a similar result for $\mathcal{E}_{\kappa, 0}$. Indeed, in this case, $\tmv \in H^1_0 (\Om, \RR^3)$ and {\eqref{eq:EnInc}}
  reads as
\[
\mathcal{E}_{\kappa, 0}( \tmmathbf{e}_3 + \vv) - \mathcal{E}_{\kappa, 0} (\tmmathbf{e}_3) = \int_{\Om} \left|\grad \vv \right|^2 - \kappa^2 \int_{\Om} | \tmmathbf{v}_{\bot} |^2. 
\]
But then the result follows from classical Poincar{\'e} inequality in $H^1_0 (\Om, \RR^3)$, by taking $\kappa_0\assign c_\Om$ where $c_\Om$ is the Poincaré constant.
\end{proof}

\begin{proof}[Proof of \Cref{prop:univconf}, \cref{universal_inplane}]
  The range of parameters under which the minimality of the constant in-plane configurations holds depends essentially on $\gamma$, and can be easily investigated through the classical trace inequality:
  \begin{equation}
    c_{\partial \Om}  \| u \|_{L^2 \left( \partial \Om \right)} \leqslant \| u
    \|_{H^1 \left( \Om \right)} \label{eq:ctrineq},
  \end{equation}
for some $c_{\partial \Om} > 0$ and every $u \in H^1 \left( \Om \right)$. Indeed, let \(\tmmathbf{e}_{\bot}\in\Stwo^2\) such that \(\tmmathbf{e}_{\bot}\cdot\tmmathbf{e}_3=0\) and let $\vv \in H^1 (\Om, \RR^3)$ such that $| \vv +\tmmathbf{e}_{\bot}|=1$. A simple computation gives that
\begin{align*}
\abs{(\vv+\tmmathbf{e}_{\bot})\times \tmmathbf{e}_{3}}^2-\abs{\tmmathbf{e}_{\bot}\times \tmmathbf{e}_{3}}^2
&=\abs{\vv\times\tmmathbf{e}_{3}}^2+2(\vv\times \tmmathbf{e}_{3})\cdot (\tmmathbf{e}_{\bot}\times\tmmathbf{e}_{3})\\
&=\abs{\vv\times\tmmathbf{e}_{3}}^2+2\vv\cdot\tmmathbf{e}_{\bot}\\
&=\abs{\vv\times\tmmathbf{e}_{3}}^2-\abs{\vv}^2\\
&=-(\vv\cdot\tmmathbf{e}_3)^2.
\end{align*}
Hence, we have
  \begin{align}
    \Ekg  (\tmmathbf{e}_{\bot} + \vv) - \Ekg (\tmmathbf{e}_{\bot}) 
    & \eqs  \int_{\Om} \left| \grad \vv \right|^2 + \kappa^2  \int_{\Om}
    \left( \vv \cdot \tmmathbf{e}_3 \right)^2 - \frac{1}{\gamma^2}
    \int_{\partial \Om} (\tmmathbf{v} \cdot \tmmathbf{e}_3)^2 \nonumber\\
    & \geqslant  \int_{\Om} \left| \grad \vv_{\bot} \right|^2 + \left(
    c^2_{\partial \Om} \cdot \min \{ 1, \kappa^2 \} - \frac{1}{\gamma^2}
    \right)  \int_{\partial \Om} \left( \vv \cdot \tmmathbf{e}_3 \right)^2,
    \nonumber
  \end{align}
  where $\vv_{\bot} =\tmmathbf{v}- \left( \tmmathbf{v} \cdot \e_3 \right)\e_3$. Therefore, as soon as 
  \[
  \gamma\geqslant\gamma_{\kappa} \assign \frac{1}{c_{\partial \Om} \cdot \min \{ 1,\kappa\}}, 
  \]
we obtain that \(\tmmathbf{e}_\bot\) is a global minimizer of \(\Ekg\). Moreover, if \(\gamma>\gamma_{\kappa}\), we have that the constant in-plane vector fields \(\tmmathbf{e}_\bot\in\Stwo^2\), with \(\tmmathbf{e}_{\bot}\cdot\tmmathbf{e}_3=0\), are the only minimizers of \(\Ekg\). Indeed, if $\Ekg(\tmmathbf{e}_{\bot} + \vv) -\Ekg (\tmmathbf{e}_{\bot}) = 0$ then $\tmv_\bot$ is constant a.e.\ in \(\Om\) and, therefore, so is $(\tmv\cdot\e_3)$ due to constraint $|\tmmathbf{e}_{\bot}+\vv|=1$ imposed on $\vv$. Since  \(\tmv\cdot\tmmathbf{e}_3=0\) a.e.\ on \(\partial\Om\), we conclude that $\tmv$ is constant and in-plane. This concludes the proof.
\end{proof}

\section{Symmetries in the target space and range of minimizers}\label{sec:3}

In this section we show that due to symmetry of the problem the range of any minimizer is contained in a meridian of \(\Stwo^2\).

\subsection{Symmetries of the energy functional in the target space}First, it
is clear that the energy is invariant under the group of isometries that
preserve the vertical coordinate axis $\RR \e_3$, i.e.,
\[ \Othreethree \assign \left\{ \sigma \in O (3) \st \sigma \left( \e_3
   \right) = \e_3 \text{ or } \sigma \left( \e_3 \right) = - \e_3 \right\}. \]
This group is generated by the isotropy group $\{\sigma \in O (3) \st \sigma
(\e_3) = \e_3 \}$ and the reflection $\sigma_{\e_3}$ through the plane
orthogonal to $\e_3$.

\begin{proposition}\label{symmetryGroup}
  For every $\kappa,\gamma \in [0, + \infty)$, $\sigma \in
  \Othreethree$ and $\m \in H^1 (\Om, \Stwo^2)$ we have $\Ekg (\m) = \Ekg (\sigma \circ
  \m)$.
\end{proposition}

\Cref{symmetryGroup} applies in particular to the reflection $\sigma=\sigma_{\tmv}$, defined by $\sigma_{\tmv} (\ww) = \ww - 2 (\vv \cdot \ww) \tmv$, through the plane orthogonal to a vector $\tmv \in \Stwo^2$ which is either equal to $\e_3$ or orthogonal to $\e_3$. Using the fact that the $H^1$ seminorm is preserved by taking the positive or negative parts, we also have the following result.

\begin{proposition}
  \label{symmetryplus} Let $\kappa \in [0,+\infty)$, $\vv \in \Stwo^2$ and $\m
  \in H^1 (\Om, \Stwo^2)$. If either $\tmv = \e_3$ or $\tmv \cdot \e_3 = 0$,
  then $\mathcal{E}_{\kappa, \gamma} (\m) =\mathcal{E}_{\kappa, \gamma}
  (\sigma_{\tmv}^+ \circ \m)$, where
  \begin{equation}
    \sigma_{\tmv}^+ (\w) \assign \left\{\arraycolsep=1.4pt\def\arraystretch{1.6}\begin{array}{ll}
      \w & \text{if } \w \cdot \tmv \geqslant 0,\\
      \w - 2 (\tmv \cdot \w) \tmv & \text{if } \w \cdot \tmv < 0.
    \end{array}\right.
  \end{equation}
\end{proposition}

This applies for instance to $\left( \sigma_{\e_1}^+ \circ \m \right) =
(\abs{m_1}, m_2, m_3)$, $\left( \sigma_{\e_2}^+ \circ \m \right) = (m_1,
\abs{m_2}, m_3)$ and $\left( \sigma_{\e_3}^+ \circ \m \right) = (m_1, m_2,
\abs{m_3})$.

\subsection{Regularity of minimizers}
For a smooth bounded domain $\Om \subset \RR^2$, the regularity of minimizers follows from the classical regularity theory of Schoen-Uhlenbeck \cite{SchUhl}. However the regularity in dimension \(N\geqslant 3\) is not trivially guaranteed in our problem, as there may exist singular minimizing homogeneous harmonic maps into \(\Stwo^2\) such as \(x\mapsto\frac{x}{\abs{x}}\) in \(\RR^3\). Here, we can prove regularity by using the symmetries. We start with an easy lemma.
\begin{lemma}
\label{sobolevContinuous}
Let \(u\in W^{1,p}(\Omega)\) be a Sobolev function defined on an open set \(\Omega\subset\RR^N\), \(p\geqslant 1\). If \(\abs{u}\) is continuous, then \(u\) is continuous. 
\end{lemma}
\begin{proof}
If \(u(x)=0\), then \(u\) is continuous at \(x\). If \(u(x)\neq 0\), then, as \(\abs{u}\) is continuous, there 
exists a non empty ball \(B_r(x)\subset\Omega\) where \(\abs{u}\geqslant\alpha>0\). Let \(v\in W^{1,p}(B_r(x))\) be defined by \(v(x)\assign\max\{\min\{\frac{1}{\alpha}\, u(x),1\},-1\}\). We 
have that \(v(x)\in\{-1,1\}\) everywhere in \(B_r(x)\), which for a Sobolev function means that \(v\) is equal to a constant a.e.\ in \(B_r(x)\). This means that the sign of \(u\) does not change on \(B_r(x)\), i.e. that \(u=\abs{u}\) a.e.\ in \(B_r(x)\) or \(u=-\abs{u}\) a.e.\ in \(B_r(x)\). Thus, \(u\) is continuous.
\end{proof}
\begin{proposition}
\label{regularityProposition}
Let \(\kappa,\gamma\in[0,+\infty)\) and let \(\m\in H^1(\Omega,\Stwo^2)\) be a global minimizer of \(\Ekg\). Then \(\m\in\mathcal{C}^\infty(\Omega,\Stwo^2)\).
\end{proposition}
\begin{remark}
Below, in the proof of \Cref{structureMinimizers}, we show that actually global minimizers of \(\Ekg\) are in \(\mathcal{C}^\infty(\overline\Omega,\Stwo^2)\), i.e., smooth up to the boundary.
\end{remark}

\begin{proof}
By \Cref{symmetryplus}, \(\tilde{\m}\assign (\abs{m_1},\abs{m_2},\abs{m_3})\) is still a global minimizer of \(\Ekg\). In particular, \(\tilde{\m}\) is a global minimizer of \(\mathcal{E}_\kappa\) under its own boundary condition. Since \(\tilde{\m}\) is valued into a strictly convex subset of the sphere \(\Stwo^2\) and since \(\mathcal{E}_\kappa\) is nothing but a perturbation of the Dirichlet energy by a lower order term (namely, the zero-order term of energy density \(\kappa^2 (\m\cdot \e_3)^2\)), we deduce from \cite[Theorem IV and its corollary]{SchUhl} that \(\tilde{\m}\) is continuous in \(\Omega\)\footnote{Note that the Shoen-Uhlenbeck regularity theory gives smoothness of \(\m\) with no restriction on the image of \(\m\) in dimension \(N=2\); in dimension \(N\geqslant 3\), the presence of singularities is ruled out thanks to the condition that \(\tilde{\m}\) is valued into a strictly convex subset of \(\Stwo^2\).}. Hence \(\m\) is continuous by \Cref{sobolevContinuous}. But it is then standard to prove that \(\m\) is smooth (we refer to \cite[\S 2.10 and \S 3.2]{Simon96}, for instance).
\end{proof}

\subsection{Range of minimizers}
We start with the following consequence of the maximum principle.
\begin{lemma}
  \label{LemmaRigid} Let $\kappa,\gamma \in [0, + \infty)$ and \(\vv\in\Stwo^2\) such that either \(\vv\cdot \e_3=0\) or \(\vv\in\{-\e_3,\e_3\}\). If $\m$ is a global minimizer of $\mathcal{E}_{\kappa,\gamma}$, then either $\m \cdot \tmv \equiv 0$ in $\Om$ or $\m \cdot\tmv$ never vanishes in $\Om$.
\end{lemma}

\begin{proof}
  By \Cref{symmetryplus}, $\sigma_{\tmv}^+ \circ \m$ is still a minimizer of $\mathcal{E}_{\kappa,\gamma}$. By \Cref{regularityProposition}, $\sigma_{\tmv}^+ \circ\m$ is smooth. In particular, \(\sigma_{\tmv}^+ \circ \m\) (and not only \(\m\)) solves the Euler-Lagrange equation \eqref{eq:ELstrong}; projecting this equation on \(\vv\), we obtain that  $\left( \sigma_{\tmv}^+ \circ \m \right)\cdot \tmv = \left| \m \cdot \tmv \right| $ solves the elliptic equation
  \[ \Delta \left| \m \cdot \tmv \right| + c (x) \left| \m \cdot \tmv \right|
     = 0 \hspace{1em} \text{in } \Om, 
     \]
     with
     \[c (x) =
     \left\{\arraycolsep=4.8pt\def\arraystretch{1.6}\begin{array}{ll}
       \abs{\nabla \left( \sigma_{\tmv}^+ \circ \m \right)}^2 + \kappa^2 m_{3^{}}^2
       & \text{if } \tmv \cdot \e_3 = 0,\\
       \abs{\nabla \left( \sigma_{\tmv}^+ \circ \m \right)}^2 + \kappa^2
       (m_{3^{}}^2 - 1) & \text{if } \tmv = \e_3.
     \end{array}\right. \]
We then apply the maximum principle \cite[Theorem 2.10]{Han2011} to find that either $\m \cdot \tmv \equiv 0$ or $\m \cdot \tmv$ does not vanish in $\Om$.
\end{proof}

\begin{proof}[Proof of \Cref{structureMinimizers}]
By \Cref{regularityProposition}, \(\m\) is smooth. For the rest of the proof, we proceed in four steps.

\noindent{\it Step 1. \(\m\) is valued into a meridian.} For $\tmv\in \Stwo^2$ such that \(\tmv\cdot e_3=0\), we denote by $\Stwo^2_+(\tmv)$ the closed hemisphere directed by $\tmv$, i.e., the closed subset of $\Stwo^2$ obtained intersecting $\Stwo^2$ with the closed half-space $\{z\in\RR^3\st z\cdot\tmv\geqslant 0\}$. If $\m\equiv \pm\e_3$ in $\Om$ there is nothing to prove. If not, there exists $x_0\in\Om$ such that the projection $\m_\perp(x_0)$ of $\m(x_0)$ onto the plane orthogonal to $\e_3$ is different from zero. We set $\tmv_0\assign\m_\perp(x_0)/|\m_\perp(x_0)|$ and we claim that the target space of $\m$ is contained in the meridian passing through $\tmv_0$. By construction, we have
\[
\m(x_0)\cdot \tmv>0\quad\text{for every }\tmv\in V\assign\{\tmv\in \Stwo^2\st \tmv\cdot e_3=0,\, \tmv\cdot\tmv_0>0\}.
\]
Therefore, by \Cref{LemmaRigid} and the continuity of $\m$, we get that for every $x\in\Om$ there holds
\[
\m(x)\in \textstyle\bigcap_{v\in V}\Stwo^2_+(\tmv).
\]
As the intersection on the right-hand side is the meridian passing through $\tmv_0$ we conclude.

\noindent{\it Step 2. The image of \(\m\) is contained in a quarter of vertical great circle (or half of meridian).}
Indeed, applying again \Cref{LemmaRigid} to $\tmv = \e_3$, we obtain that either
\[
\text{\(\m\cdot \e_3\geqslant 0\) in \(\Om\),}\quad\text{ or \(\m\cdot \e_3\leqslant 0\) in \(\Om\).}
\]
 Since $\Othreethree$ acts transitively on the quadrants of meridians, we can express \(\m\) in terms of a particular solution valued into the quadrant of meridian \(\{m_2=0\}\cap\{m_1,m_3\geqslant 0\}\). Namely, there exists $\sigma \in \Othreethree$ such that $\m = \sigma \circ \tmu$, where $\tmu\in \mathcal C^\infty(\Om,\Stwo^2)$ is of the form \(\tmu=(u_1,0,u_2)\) with \(u_1,u_2\in \mathcal C^\infty(\Om,\RR)\) such that \(u_1^2+u_2^2=1\) and $u_1, u_2 \geqslant 0$ in $\Om$. We then lift the map $\tmu$ to $\RR$ by writing $\tmu = (\sin \varphi, 0, \cos \varphi)$ with $\varphi \in \mathcal C^\infty( \Om)$ and \(0\leqslant\varphi\leqslant\frac\pi 2\)  in \(\Omega\). We conclude by noticing that \Cref{LemmaRigid} also tells us that either \(\varphi\equiv 0\), or \(\varphi\equiv\frac\pi 2\), or \(0<\varphi<\frac\pi 2\) in \(\Om\).

\noindent{\it Step 3. Regularity up to the boundary.}  In the proof of Proposition~\ref{regularityProposition} we used a symmetry argument in order to apply the regularity theory of Schoen-Uhlenbeck \cite{SchUhl} and infer that a minimizer $\m \in \mathcal C^\infty(\Om,\mathbb S^2)$; we now claim that $\m\in \mathcal C^\infty(\overline\Om,\mathbb S^2)$, using the previous steps. Indeed, when $\gamma >0$ it is clear that the lifting $\varphi$ of $\m$ satisfies
\begin{align}
-\Delta \varphi &=\frac{\kappa^2}{2}\sin(2\varphi) \quad \text{in \(\Om\),}\\
\frac{\partial \varphi}{\partial \n}& =-\frac{1}{2 \gamma^2}\sin (2\varphi) \quad\text{on \(\partial \Om\)}.
\end{align}
\noindent Since $\frac{1}{\gamma^2} \sin (2 \varphi) \in H^{1 / 2} (\partial \Omega)$,
there exists $\tilde{\varphi} \in H^2(\Omega)$ such that $\partial_{\n} \tilde{\varphi} = - \frac{1}{2\gamma^2} \sin (2 \varphi)$ (see, e.g., \cite[Theorem~5.8]{Necas}). For the difference $\varphi - \tilde{\varphi}$ we have $\Delta (\varphi - \tilde{\varphi}) \in L^2(\Omega)$ and $\partial_{\n} (\varphi - \tilde{\varphi})  =  0$ on $\partial \Omega$. Hence, by classical elliptic regularity, we have $\varphi - \tilde{\varphi} \in H^2 (\Omega)$. It follows that $\varphi \in H^2 (\Omega)$ and therefore $\sin(2\varphi) \in H^2(\Omega)$ (see e.g. \cite[Proposition~3.9]{Taylor96}) and  $- \frac{1}{\gamma^2} \sin (2 \varphi)\in H^{3/2}(\partial \Omega)$. A bootstrap argument and Sobolev embedding theorems conclude the proof. Indeed, one can iterate the construction to infer the existence for every $k\geqslant 2$ of $\tilde{\varphi} \in H^k(\Omega)$ such that $\Delta (\varphi - \tilde{\varphi}) \in H^{k-2}(\Omega)$ and $\partial_{\n} (\varphi - \tilde{\varphi})  =  0$ on $\partial \Omega$. In the case $\gamma=0$ regularity up to the boundary follows from the standard elliptic regularity.

\noindent{\it Step 4. Range of \(\operatorname{tr}_{\partial\Omega}\varphi\).} We now show that when $\gamma>0$, if \(0<\varphi<\frac\pi 2\) in \(\Om\), then \(0<\varphi<\frac\pi 2\) on \(\partial\Om\). Let us assume that at some point $x_0 \in \partial \Om$ we have $\varphi (x_0)=0$. Then we know that $\Delta (-\varphi) \geqslant 0$ and  $-\varphi(x_0)> -\varphi$ in $\Omega$. Using Hopf Lemma (see \cite[Lemma 3.4]{gilbarg77}), we deduce that $\partial_{\n} (-\varphi) (x_0) >0$, which contradicts the boundary condition $\partial_{\n} \varphi (x_0) =0$. Hence $\varphi(x)>0$ in $\overline \Omega$.

Assume now that there exists $x_0 \in \partial \Omega$ such that $\varphi(x_0) = \frac{\pi}{2}$. Then we define $u =\varphi-\frac{\pi}{2} \leqslant 0$  and  we have $0=u(x_0)>u(x)$ for $x \in \Om$. Moreover, we have
$$
\Delta u - \kappa^2\frac{\sin(2u)}{2u} u =0 \quad \hbox{  in } \ \Omega.
$$
Defining $c(x) =  - \kappa^2\frac{\sin(2u(x))}{2u(x)}$ we know that $-\kappa^2 \leqslant c(x) \leqslant 0$. Therefore, using  \cite[Lemma 3.4]{gilbarg77} in the case $c(x) \leqslant 0$ we deduce that $\partial_{\n} u(x_0)>0$, which contradicts $\partial_{\n} u(x_0)=0$ due to boundary conditions. Hence $\varphi(x) < \frac{\pi}{2}$ in $\overline \Omega$.
\end{proof}

\section{Uniqueness of minimizers up to a symmetry}\label{sec:4}

{\noindent} In this section we prove \Cref{corollaryDirichlet}, which is a direct consequence of the following two lemmas.

\begin{lemma}
  \label{prop:mainprop}
Let $\m \in H^1(\Om, \Stwo^2 )$ and $\vv \in H^1_0 \left( \Om, \RR^3 \right)$ satisfy $\m+\tmmathbf{v}\in\Stwo^2$ a.e.\  in \(\Omega\). If $\m$ satisfies the Euler-Lagrange equations \eqref{eq:ELweak} and if $m_1 > 0$ a.e.\  in $\Om$, then \(m_1\) is bounded below by positive constants on compact subsets of \(\Om\) and
  \begin{equation}
    \mathcal{E}_{\kappa} (\m+\tmmathbf{v}) -\mathcal{E}_{\kappa}
    (\m) \geqslant \int_{\Om} m_1^2  \left| \nabla \left(
    \frac{\tmmathbf{v}}{m_1} \right) \right|^2 + \kappa^2  \int_{\Om}
    (\tmmathbf{v} \cdot \tmmathbf{e}_3)^2 . \label{eq:mainresulest}
  \end{equation}
\end{lemma}

\begin{proof}[Proof of \Cref{prop:mainprop}]
  We follow the ideas of \cite[Lemma A.1]{INSZ15}, \cite[Theorem 4.3]{Dif2016} and \cite[Theorem 5.1]{INSZ18}. We have
  \[ \mathcal{E}_{\kappa} (\m+\tmmathbf{v}) -\mathcal{E}_{\kappa}
     (\m) = \int_{\Om} | \nabla \vv |^2 + \kappa^2  \int_{\Om} (\tmmathbf{v}
     \cdot \tmmathbf{e}_3)^2 + 2 \int_{\Om} \nabla \m  \Fsp \nabla \vv + 2
     \kappa^2 \int_{\Om} \left( \m \cdot \tmmathbf{e}_3 \right) (\tmmathbf{v}
     \cdot \tmmathbf{e}_3) . \]
  Note that since $\abs{\m} = \abs{\m+\tmmathbf{v}} =
  1$ a.e., we also have $| \vv | \leqslant 2$ a.e.\ in $\Om$. In particular, $\tmmathbf{v} \in H^1_0 \left( \Om, \RR^3 \right) \cap L^{\infty}
  (\Om, \RR^3)$. Since $\m$ satisfies the Euler-Lagrange equations
  {\eqref{eq:ELweak}}, we get
  \[ \mathcal{E}_{\kappa} (\m+\tmmathbf{v}) -\mathcal{E}_{\kappa}
     (\m) = \int_{\Om} | \nabla \vv |^2 + \kappa^2  \int_{\Om} (\tmmathbf{v}
     \cdot \tmmathbf{e}_3)^2 + 2 \int_{\Om} (| \nabla \m |^2 + \kappa^2 
     \left( \m \cdot \tmmathbf{e}_3 \right)^2) \m \cdot \tmmathbf{v}. \]
  On the other hand, since $\abs{\m+\tmmathbf{v}} = 1$, we have $ 2
  \m \cdot \vv =-| \vv |^2$ and, therefore,
  \begin{equation}
    \mathcal{E}_{\kappa} (\m+\tmmathbf{v}) -\mathcal{E}_{\kappa}
    (\m) = \int_{\Om} | \nabla \vv |^2 + \kappa^2  \int_{\Om} (\tmmathbf{v}
    \cdot \tmmathbf{e}_3)^2 - \int_{\Om} (| \nabla \m |^2 + \kappa^2  \left(
    \m \cdot \tmmathbf{e}_3 \right)^2) | \tmmathbf{v} |^2 .
  \end{equation}

Hence, \eqref{eq:mainresulest} will follow once we prove that for all $\vv \in H^1_0\cap L^\infty \left( \Om, \RR^3 \right)$,
  \begin{equation}\label{linearEstimate}
  \int_{\Om} | \nabla \vv |^2 \geqslant \int_{\Om} (| \nabla \m |^2 + \kappa^2  \left(
    \m \cdot \tmmathbf{e}_3 \right)^2) | \tmmathbf{v} |^2
    + \int_{\Om} m_1^2  \left| \nabla \left(
    \frac{\tmmathbf{v}}{m_1} \right) \right|^2.
  \end{equation}
  We first assume that $\tmmathbf{v} \in C_c^{\infty} (\Om, \RR^3)$, the general case will follow by density. 

  Now, by the Euler-Lagrange equation of \(m_1\) in \eqref{eq:ELweak} and since \(m_1\) is assumed to be positive in \(\Om\), we have in particular that \(m_1\) is a positive weak superharmonic function, i.e. \(\Delta m_1\leqslant 0\) weakly in \(\Om\); we deduce from the weak Harnack-Moser inequality (see \cite[Theorem 14.1.2.]{jostPartialDifferentialEquations2013}) that \(m_1\) is bounded from below by a positive constant on the support of $\tmv$. Hence, we can write $\tmv$ in the form
  \begin{equation}
    \tmmathbf{v}_{} = m_1 \tmmathbf{u},
  \end{equation}
  where $\tmmathbf{u}= \frac{\tmmathbf{v}}{m_1} \in H_0^1 \left( \Om, \RR^3
  \right) \cap L^{\infty} (\Om, \RR^3)$. We then compute
  \begin{align}
    \int_{\Om} | \nabla \vv |^2 & \eqs  \sum_{j = 1}^N \int_{\Om}
    |\tmmathbf{u} \partial_j m_1 + m_1 \partial_j \tmmathbf{u}|^2 \\
    & \eqs  \int_{\Om} | \tmmathbf{u} |^2  \left| \grad m_1 \right|^2 +
    m_1^2  \left| \grad \tmmathbf{u} \right|^2 + m_1 \grad m_1 \cdot \grad |
    \tmmathbf{u} |^2 \\
    & \eqs  \int_{\Om} m_1^2  \left| \grad \tmmathbf{u} \right|^2 + \grad
    m_1 \cdot \grad (m_1 | \tmmathbf{u} |^2) .  \label{eq:exprgradv}
  \end{align}
  Now, testing the Euler-Lagrange equations \eqref{eq:ELweak} against $\tmmathbf{\varphi}
  \assign m_1 | \tmmathbf{u} |^2 \tmmathbf{e}_1 \in H^1_0 \left( \Om, \RR^3
  \right) \cap L^{\infty} (\Om, \RR^3)$, we obtain
  \begin{equation}
    \int_{\Om} \grad m_1 \cdot \grad (m_1 | \tmmathbf{u} |^2) = \int_{\Om} (|
    \nabla \m |^2 + \kappa^2  \left( \m \cdot \tmmathbf{e}_3 \right)^2) m_1^2
    | \tmmathbf{u} |^2 .
  \end{equation}
  Combining the previous two relations, and recalling that $\tmmathbf{v}= m_1
  \tmmathbf{u}$, we obtain the following identity:
  \begin{equation}
    \int_{\Om} | \nabla \vv |^2 = \int_{\Om} m_1^2  \left| \grad \tmmathbf{u}
    \right|^2 + (| \nabla \m |^2 + \kappa^2  \left( \m \cdot \tmmathbf{e}_3
    \right)^2) | \tmmathbf{v} |^2 . \label{secondEstimateDif}
  \end{equation}
This proves \eqref{linearEstimate} in the case where $\tmmathbf{v} \in C_c^{\infty}(\Om, \RR^3)$. In general, we have $\vv \in H^1_0 \left( \Om, \RR^3 \right)
  \cap L^{\infty} (\Om, \RR^3)$ and there thus exists a sequence $\left( \vv_n
  \right)_{n \in \mathbb{N}} $ in $C_c^{\infty} (\Om, \RR^3)$ such that
  \[
  \sup_{n \in \NN} \| \vv_n \|_{\infty} \leqslant \| \vv \|_{\infty} + 1
  \] 
  and
  \begin{equation}
    \tmmathbf{v}_n \rightarrow \tmmathbf{v} \quad \text{in } H_0^1 \left( \Om,
    \RR^3 \right) .
  \end{equation}
  By the previous computations in the smooth case, we have for every compact
  $K \subset \Om$ and $n \in \NN$,
  \[ 
  \int_{\Om} | \nabla \vv_n |^2 \geqslant \int_{K} (| \nabla \m |^2 + \kappa^2  \left(
    \m \cdot \tmmathbf{e}_3 \right)^2) | \tmmathbf{v}_n |^2
    + \int_{K} m_1^2  \left| \nabla \left(
    \frac{\tmmathbf{v}_n}{m_1} \right) \right|^2.  
  \]
  The conclusion follows by passing to the limit $n \rightarrow \infty$ using
  the dominated convergence theorem, and then taking the supremum over
  compacts $K \subset \Om$ using the monotone convergence theorem.
\end{proof}

\begin{lemma}
  \label{prop:mainprop1}
Let $\m \in H^1(\Om, \Stwo^2 )$ and $\vv \in H^1 \left( \Om, \RR^3 \right)$ satisfy $\m+\tmmathbf{v}\in\Stwo^2$ a.e.\  in \(\Omega\). If $\m$ satisfies the Euler-Lagrange equations \eqref{eq:ELweak} and if $m_1, m_3 > 0$  in $\overline\Om$, then 
  \begin{equation}
    \Ekg (\m+\tmmathbf{v}) -\Ekg
    (\m) \geqslant  \int_{\Om} m_1^2  \abs{\grad \tmmathbf{u}^\perp}^2 + m_3^2  \abs{\grad u_3}^2,
  \end{equation}
where we wrote $\vv = m_1 \tmu^\perp +  m_3 u_3 \e_3$ with $\tmu =(\tmu^\perp, u_3) \in H^1(\Omega, \RR^3)$.
\end{lemma}

\begin{proof}[Proof of \Cref{prop:mainprop1}]
  As in Lemma~\ref{prop:mainprop} we have
  \begin{align}
   \Ekg (\m+\tmmathbf{v}) -\Ekg
     (\m) &= \int_{\Om} | \nabla \vv |^2 + \kappa^2  \int_{\Om} (\tmmathbf{v}
     \cdot \tmmathbf{e}_3)^2 + 2 \int_{\Om} \nabla \m  \Fsp \nabla \vv + 2
     \kappa^2 \int_{\Om} \left( \m \cdot \tmmathbf{e}_3 \right) (\tmmathbf{v}
     \cdot \tmmathbf{e}_3) \notag \\
     &\qquad - \frac{1}{\gamma^2} \int_{\partial\Omega}(\vv \cdot \tmmathbf{e}_3)^2 + 2 (\m \cdot \tmmathbf{e}_3) (\vv \cdot \tmmathbf{e}_3) 
   \end{align}
and $\tmmathbf{v} \in H^1 \left( \Om, \RR^3 \right) \cap L^{\infty}
  (\Om, \RR^3)$. Since $\m$ satisfies the Euler-Lagrange equations
  {\eqref{eq:ELweak}}
and $ 2  \m \cdot \vv =-| \vv |^2$ we obtain
  \begin{align}
    \Ekg (\m+\tmmathbf{v}) -\Ekg
    (\m) &= \int_{\Om} | \nabla \vv |^2 + \kappa^2  \int_{\Om} (\tmmathbf{v}
    \cdot \tmmathbf{e}_3)^2 - \int_{\Om} (| \nabla \m |^2 + \kappa^2  \left(
    \m \cdot \tmmathbf{e}_3 \right)^2) | \tmmathbf{v} |^2 \notag \\
        &\qquad-\frac{1}{\gamma^2} \int_{\partial\Omega}(\vv \cdot \tmmathbf{e}_3)^2 -  \left( \m \cdot \tmmathbf{e}_3 \right)^2 |\vv|^2. 
  \end{align}
Now we use the fact $m_1>0$, $m_3>0$ in $\overline\Omega$ and represent $\vv = m_1 \tmu^\perp +  m_3 u_3 \e_3$ with $\tmu =(\tmu^\perp, u_3) \in H^1(\Omega, \RR^3)$. In that case, we have (in the following computations, we assume that \(\m\) is smooth; the general case follows similarly, as in the proof of \Cref{prop:mainprop})
 \begin{align}
    \int_{\Om} | \nabla \vv |^2 &= \int_{\Om} m_1^2  \abs{\grad \tmmathbf{u}^\perp}^2 + \grad
    m_1 \cdot \grad (m_1 | \tmmathbf{u}^\perp |^2)   + m_3^2 \abs{\nabla u_3}^2 + \nabla m_3 \cdot \nabla (m_3 u_3^2) \\\
    &= \int_{\Om} m_1^2  \abs{\grad \tmmathbf{u}^\perp}^2 - \Delta
    m_1 \  m_1 \abs{\tmmathbf{u}^\perp}^2 + \int_{\partial \Om} \partial_{\n} m_1 \ m_1 \abs{\tmu^\perp}^2 \notag\\
    &\qquad+\int_{\Om} m_3^2  \abs{\grad u_3}^2 - \Delta
    m_3 \  m_3\, u_3^2 + \int_{\partial \Om} \partial_{\n} m_3 \ m_3 \, u_3^2 \\
     &=  \int_{\Om} m_1^2  \abs{\grad \tmmathbf{u}^\perp}^2 + m_3^2  \abs{\grad u_3}^2 +   
    (| \nabla \m |^2 + \kappa^2  \left( \m \cdot
    \tmmathbf{e}_3 \right)^2) |\vv|^2 - \kappa^2 (\vv \cdot \e_3)^2\notag  \\
     &\qquad + \frac{1}{\gamma^2} \int_{\partial \Om}  (\vv \cdot \e_3)^2 - m_3^2 |\vv|^2 .
      \end{align}
Plugging it into the energy difference we obtain
  \begin{align}
    \Ekg (\m+\tmmathbf{v}) -\Ekg
    (\m) = \int_{\Om} m_1^2  \abs{\grad \tmmathbf{u}^\perp }^2 + m_3^2  \abs{\grad u_3}^2.
  \end{align}
  This concludes the proof.
\end{proof}

\begin{proof}[Proof of \Cref{corollaryDirichlet}]
We first consider the case $\gamma=0$. If the constant out-of-plane configurations \(\pm \e_3\) are the only global minimizers of \(\mathcal{E}_{\kappa,0}\), we are done. If not, this means by \Cref{structureMinimizers} that \(\mathcal{E}_{\kappa,0}\) has a global minimizer of the form \( \m=(\sin\varphi,0,\cos\varphi)\) with \(\varphi\in H^1(\Om)\) such that \(0<\varphi\leqslant\frac\pi 2\) a.e.\  in \(\Om\). If \(\Bar\m=\m+\vv\) is another minimizer with \(\vv\in H^1_0(\Om,\RR^3)\), then we have by \Cref{prop:mainprop} that \(\vv=m_1 \vv_0\) for some \(\vv_0\in\RR^3\). But, in order to satisfy the constraint \(\m+m_1\vv_0\in\Stwo^2\), we must have \(\vv_0\cdot (m_1\vv_0+2\m)=0\). Restricted to the boundary \(\partial\Om\), where we have \(\m=\e_3\), this condition yields \(\vv_0\cdot \e_3=0\). Hence, since \(m_2\equiv 0\), we arrive at the equation \(0=\vv_0\cdot (m_1\vv_0+2m_1 e_1)\) which means that \(\abs{\vv_0+e_1}^2=1\). Hence, \(\vv_0=(\cos \theta -1,\sin\theta,0)\) for some \(\theta\in\RR\), which means that \(\Bar\m=(\cos\theta \sin\varphi,\sin\theta\sin\varphi,\cos\varphi)\), i.e. \(\m\) is a rotation of \(\m\) of angle \(\theta\) around the \(x_3\)-axis.

\vskip 0.2cm

In the case $\gamma>0$ if constant in-plane or constant out-of-plane $\pm \e_3$ configurations are the only minimizers then the result follows by noting that $\pm \e_3$ and constant in-plane configurations cannot be minimizers simultaneously. Indeed, assuming that their energies coincide and are optimal, then we would have that the constant vectors $\m=(a_1, 0, a_3)$, with $a_1, a_3>0$ such that \(a_1^2+a_3^2=1\), have the same energy; but these configurations do not solve the Euler-Lagrange equations, contradicting the minimality.
 
If constant configurations are not the only minimizers, then by \Cref{structureMinimizers} the energy \(\mathcal{E}_{\kappa,\gamma}\) has a global minimizer of the form \( \m=(\sin\varphi,0,\cos\varphi)\) with \(\varphi\in H^1(\Om)\) such that \(0<\varphi<\frac\pi 2\) \  in \(\overline\Om\). If \(\Bar\m=\m+\vv\) is another minimizer with \(\vv\in H^1(\Om,\RR^3)\), then we have by \Cref{prop:mainprop1} that \(\vv=m_1 \ww^\perp + m_3 w_3 \mathbf \e_3 \) for some \(\ww=(\ww^\perp, w_3) \in\RR^3\). But, in order to satisfy the constraint \(\m+\vv\in\Stwo^2\), we must have $w_3 \in \{0, -2\}$ and $|\ww^\perp + \e_1| =1$, implying the result.
\end{proof}

\begin{remark} \label{rmk:unique}
We note that using Lemma~\ref{prop:mainprop} we actually proved the uniqueness (up to a symmetry) of a solution of the Euler-Lagrange equation \eqref{eq:ELstrong} with $m_1>0$ under boundary condition $\m=\pm \e_3$ for any $\kappa>0$. In a  similar way it is possible to prove the same result under a prescribed Dirichlet boundary conditions
$\tmmathbf g\in H^{1 / 2} \left( \partial \Om, \Stwo^2 \right)$ with either $g_1>0$ or $g_2>0$. 

Analogously, using Lemma~\ref{prop:mainprop1} we have proved the uniqueness (up to a symmetry) of a solution of the Euler-Lagrange equation \eqref{eq:ELweak} with $m_1>0$ and $m_3>0$ on $\overline \Om$ for any $\gamma>0$ and $\kappa>0$.
\end{remark}

\section{Comparison of solutions} \label{sec:compsols}

\subsection{Minimizers with different penalizations}
Before proving 
\Cref{theoremComparison}, we introduce the localized energy functional defined for \(\kappa\geqslant 0\), \(\gamma>0\), $\varphi \in H^1 (\Omega, \RR)$ and every Borel set \(O\subset\RR^N\), by
\begin{equation}
\mathcal{F}_{\kappa,\gamma}(\varphi,O)\assign \int_{O\cap\Omega}\abs{\nabla \varphi}^2 + \kappa^2  \int_{O\cap\Omega} \cos^2 \varphi +\frac{1}{\gamma^2}\int_{O\cap\partial\Omega}\sin^2\varphi.
\end{equation}
When \(\gamma=0\), we also define for every Borel subset \(O\subset\Omega\),
\begin{equation}
\mathcal{F}_{\kappa,0}(\varphi,O)\assign \int_{O}\abs{\nabla \varphi}^2 + \kappa^2  \int_{O} \cos^2 \varphi.
\end{equation}
\begin{proof}[Proof of \Cref{theoremComparison}]
We proceed in 5 steps. 

\noindent{\it Step 1. Energy estimates.}
Let $\varphi_i$ be a minimizer of the energy $\mathcal{E}_{\kappa_i,\gamma_i}$ valued into \([0,\frac \pi 2]\), with \(i=1,2\). By Theorem~\ref{structureMinimizers} we know that \(\varphi_1\) and \(\varphi_2\) are smooth on \(\overline\Omega\). We define 
\[
O\assign\{x\in\overline\Omega\st\varphi_1(x)>\varphi_2(x)\};
\]
we shall see that \(O=\emptyset\), i.e., \(\varphi_1\leqslant\varphi_2\) on \(\overline\Omega\). First, observe that, using  \(\nabla\varphi_1=\nabla\varphi_2\) a.e.\ on \(\{\varphi_1=\varphi_2\}\cap\Omega\) and comparing the functions \(\varphi_1,\varphi_2\) with their minimum \(\varphi_1\wedge\varphi_2\) and their maximum \(\varphi\vee\varphi_2\), by minimality,
\begin{equation}
\label{minimality_MinMax}
\mathcal{F}_{\kappa_1,\gamma_1}(\varphi_1,\Bar\Omega)\leqslant \mathcal{F}_{\kappa_1,\gamma_1}(\varphi_1\wedge\varphi_2,\Bar\Omega)
\quad
\text{and}
\quad
\mathcal{F}_{\kappa_2,\gamma_2}(\varphi_2,\Bar\Omega)\leqslant \mathcal{F}_{\kappa_2,\gamma_2}(\varphi_1\vee\varphi_2,\Bar\Omega).
\end{equation}

Hence, using \eqref{minimality_MinMax}, we obtain the following estimates (where if \(\gamma_1=0\), so that \(\varphi_1\in H^1_0(\Omega)\) and \(O\subset\Omega\), all the boundary terms are ignored),
\begin{align*}
\mathcal{F}_{\kappa_1,\gamma_1}(\varphi_1,O)&\leqslant
\mathcal{F}_{\kappa_1,\gamma_1}(\varphi_2,O)\\
&=\mathcal{F}_{\kappa_2,\gamma_2}(\varphi_2,O)+\left(\kappa_1^2-\kappa_2^2\right)\int_{O\cap\Omega}\cos^2\varphi_2+\left(\frac{1}{\gamma_1^2}-\frac{1}{\gamma_2^2}\right)\int_{O\cap\partial\Om}\sin^2\varphi_2\\
&\leqslant \mathcal{F}_{\kappa_2,\gamma_2}(\varphi_1,O)+(\kappa_1^2-\kappa_2^2)\int_{O\cap\Omega}\cos^2\varphi_2+\left(\frac{1}{\gamma_1^2}-\frac{1}{\gamma_2^2}\right)\int_{O\cap\partial\Om}\sin^2\varphi_2\\
&= \mathcal{F}_{\kappa_1,\gamma_1}(\varphi_1,O)+(\kappa_1^2-\kappa_2^2)\int_{O\cap\Omega}\left(\cos^2\varphi_2-\cos^2\varphi_1\right)\\
&\qquad+\left(\frac{1}{\gamma_1^2}-\frac{1}{\gamma_2^2}\right)\int_{O\cap\partial\Om}\left(\sin^2\varphi_2-\sin^2\varphi_1\right),
\end{align*}
from which we infer that
\begin{equation*}
(\kappa_1^2-\kappa_2^2)\int_{O\cap\Omega}\bigl(\cos^2\varphi_2-\cos^2\varphi_1\bigr)+\left(\frac{1}{\gamma_1^2}-\frac{1}{\gamma_2^2}\right)\int_{O\cap\partial\Om}\bigl(\sin^2\varphi_2-\sin^2\varphi_1\bigr)\geqslant 0.
\end{equation*}
Now we use the assumption \(\kappa_1\leqslant\kappa_2\), \(\gamma_1\leqslant\gamma_2\) and \((\kappa_1,\gamma_1)\neq (\kappa_2,\gamma_2)\). Noticing that \(x\mapsto\cos^2(x)\) is decreasing and $x\mapsto\sin^2(x)$ is increasing on \([0,\frac \pi 2]\), and that \(\varphi_2<\varphi_1\) on \(O\) by definition, it entails
\begin{equation}
\label{comparison_energy}
\begin{cases}
O\cap\Omega=\emptyset&\text{if \(\kappa_1<\kappa_2\),}\\
O\cap\partial\Omega=\emptyset&\text{if \(\gamma_1<\gamma_2\)}.
\end{cases}
\end{equation}
(Note that when \(\gamma_1=0\), the second assertion is not a consequence of the previous estimates, but it is trivially satisfied because \(\varphi_1=0\) on \(\partial\Omega\) and \(O\subset\Omega\) in this case.)

The first assertion means that
\begin{equation}
\label{comparison_kappaLess}
\varphi_1\leqslant\varphi_2\quad \text{in \(\overline\Omega\) if \(\kappa_1<\kappa_2\).}
\end{equation}

When \(\kappa_1=\kappa_2\), we need more work:

\noindent{\it Step 2. Comparison of solutions when \(\kappa_1=\kappa_2=:\kappa\) and \(\gamma_1<\gamma_2\).} By \eqref{comparison_energy}, we have
\begin{equation}
\label{comparison_boundary}
\varphi_1\leqslant\varphi_2\quad\text{on }\partial\Omega.
\end{equation}
Hence, \(\varphi_1=\varphi_1\wedge\varphi_2\) and \(\varphi_2=\varphi_1\vee\varphi_2\) on \(\partial\Omega\). By minimality (cf.~Remark~\ref{rmk:unique}) this yields
\[
\mathcal{E}_{\kappa}(\varphi_1)\leqslant \mathcal{E}_{\kappa}(\varphi_1\wedge\varphi_2)
\quad
\text{and}
\quad
\mathcal{E}_{\kappa}(\varphi_2)\leqslant \mathcal{E}_{\kappa}(\varphi_1\vee\varphi_2).
\]
Since we have also \(\mathcal{E}_{\kappa}(\varphi_1)+\mathcal{E}_{\kappa}(\varphi_2)=\mathcal{E}_{\kappa}(\varphi_1\wedge\varphi_2)+\mathcal{E}_{\kappa}(\varphi_1\vee\varphi_2)\), we obtain that
\[
\mathcal{E}_{\kappa}(\varphi_1)= \mathcal{E}_{\kappa}(\varphi_1\wedge\varphi_2)
\quad
\text{and}
\quad
\mathcal{E}_{\kappa}(\varphi_2)= \mathcal{E}_{\kappa}(\varphi_1\vee\varphi_2).
\]
The first equality means that both \(\varphi_1\) and \(\varphi_1\wedge\varphi_2\) minimize \(\mathcal{E}_\kappa\) under their own Dirichlet boundary condition on \(\partial\Omega\); in particular, they solve the Euler-Lagrange equation \(-\Delta u=\kappa^2\sin u\) in the variable \(u=2\varphi\). By uniqueness for positive solutions to sublinear elliptic equations with Dirichlet boundary conditions (see \cite[Appendix II]{brezisSublinearEllipticEquations1992}), we deduce that either \(\varphi_1\equiv 0\), or \(\varphi_1\wedge\varphi_2\equiv 0\), or \(\varphi_1\equiv \varphi_1\wedge\varphi_2\). In the case where \(\varphi_1\wedge\varphi_2\equiv 0\) but \(\varphi_1\) is not identically \(0\), since \(\Delta\varphi_2\leqslant 0\), we deduce from the strong maximum principal that \(\varphi_2\equiv 0\); in particular, \(\varphi_1=0\) on \(\partial\Omega\) by \eqref{comparison_boundary}, and we deduce from our uniqueness result for minimizers under homogeneous Dirichlet boundary conditions \Cref{corollaryDirichlet}, that \(\varphi_1\equiv\varphi_2(\equiv 0)\) on \(\Om\). In any case, we have proved that
\begin{equation}
\label{comparison_sameKappa}
\varphi_1\leqslant \varphi_2\quad \text{on \(\overline\Om\) if \(\gamma_1<\gamma_2\)}.
\end{equation}

\noindent{\it Step 4. Strict comparison of solutions in \(\Omega\).}
We know that \(u_i=2\varphi_i\) solve the Euler--Lagrange equations associated with \(\mathcal{E}_{\kappa_i,\gamma_i}\),
\begin{align}
-\Delta u_i &=\kappa_i^2\sin(u_i) \quad \text{in \(\Om\),}\label{inOmega}\\
\frac{\partial u_i}{\partial \n} &=-\frac{1}{\gamma_i^2}\sin u_i \quad\text{on \(\partial \Om\) if \(0<\gamma_i\) (and \(u_i\equiv 0\) on \(\partial\Omega\) if \(\gamma_i=0\))},\label{onBoundary}
\end{align}
and that, by \eqref{comparison_kappaLess} and \eqref{comparison_sameKappa},
\begin{equation}
u_1\leqslant u_2\quad\text{in \(\overline\Omega\).}
\end{equation}
Using that \(u\mapsto u+\sin u\) is non decreasing and that \(\kappa_1\leqslant\kappa_2\), we find that
\begin{equation}
\label{strongMaximum}
-\Delta(u_2-u_1)+\kappa_1^2(u_2-u_1)\geqslant\kappa_1^2(\sin (u_2)+u_2-\sin(u_1)-u_1)\geqslant 0\quad\text{in \(\Om\)}.
\end{equation}
By the strong maximum principal, we obtain that either \(u_2\equiv u_1\) in \(\Om\), or \(u_1<u_2\) in \(\Om\).

In the first case, i.e., \(u_1\equiv u_2\), we deduce by \eqref{inOmega} that if \(\kappa_1<\kappa_2\) then either \(u_1=u_2\equiv 0\) or \(u_1=u_2\equiv\pi\) in \(\Omega\); when \(\kappa_1=\kappa_2\) and \(\gamma_1<\gamma_2\), we have by \eqref{onBoundary} that either \(u_1=u_2\equiv\pi\) or \(u_1=u_2\equiv 0\) on \(\partial\Omega\) and by Hopf lemma it follows that 
\(u_1\) and $u_2$ are constants in \(\Omega\) since \(\frac{\partial u_i}{\partial \n}\equiv 0\) on \(\partial\Omega\).

\noindent{\it Step 5. Strict comparison of solutions on \(\partial\Omega\) when \(u_1<u_2\) in \(\Om\) and $\gamma_2>0$.}
In this case, we want to show that  \(u_1<u_2\) on \(\partial\Om\). Assume there is $x_0 \in \partial \Omega$ such that $u_1(x_0)=u_2(x_0)$ then by Hopf lemma we obtain \(\frac{\partial (u_2-u_1) }{\partial \n}(x_0)<0\). If \(0<\gamma_1\leqslant\gamma_2\) then using \eqref{onBoundary} we deduce that 
$\frac{\partial (u_2-u_1) (x_0)}{\partial \n} \geqslant 0$, obtaining the contradiction.  If \(\gamma_1=0<\gamma_2\), we know that \(u_1=0\) on \(\partial\Omega\); hence $u_2(x_0)=0$ and, by \eqref{onBoundary}, \(\frac{\partial u_2}{\partial \n}(x_0)=0\). Applying Hopf lemma to $u_2$ we obtain that \(u_2\equiv 0\) in \(\overline\Omega\), thus contradicting the fact that \(u_1<u_2\) in \(\Om\).  Therefore if $\gamma_2>0$ we have \(u_2>u_1\) on \(\overline\Om\).
\end{proof}

\section{Radial symmetry of minimizers in a ball: Proof of
\texorpdfstring{\Cref{prop:symm}}{}}\label{sec:5}
Numerical simulations suggest that when the domain $\Om$ has spherical
symmetry, the minimizers of $\Ekg$ are radially symmetric
({\tmabbr{cf.}}~Figure \ref{fig:1}). The aim of this section is to turn this
observation into a quantitative statement.

The proof we give below for the radial symmetry of minimizers $\Ekg$ also
works for the boundary value problem associated with $\mathcal{E}_{\kappa,
0}$. However, radiality of the minimizers of $\mathcal{E}_{\kappa, 0}$
immediately follows from a celebrated result of Gidas-Ni-Nirenberg
{\cite{Gidas1979}} about radial symmetry for semilinear elliptic equations. We
give the details below.

\begin{proposition}
  \label{corollaryDirichletRadial} If $\Om$ is a ball centered at the origin, then any minimizer $\m$ of the energy $\mathcal{E}_{\kappa, 0}$ is radially symmetric. More precisely, \(\m\) is either constant with \(\m\cdot \e_3\in\{0,-1,1\}\), or there exist $\sigma \in \Othreethree$ and a solution ${\varphi} :\RR^+ \to (0,\frac\pi 2)$ in \eqref{phaseDirichletProblem} such that
  \[ \m (x) = \sigma \circ (\sin {\varphi} (\abs{x}), 0, \cos
     {\varphi} (\abs{x})) \quad \text{a.e.\  in } \Om . \]
\end{proposition}

\begin{proof}[Proof of \Cref{corollaryDirichletRadial}]
  Without loss of generality, one can assume that $\m$ is not constant. By
  \Cref{structureMinimizers}, there exists $\sigma \in
  \Othreethree$ and a solution $\varphi \in H_0^1( \Om) \cap \mathcal C^\infty(\overline\Om)$ of \eqref{phaseDirichletProblem} such that $\m = \sigma (\sin \varphi, 0,\cos \varphi)$ and \(0<\varphi<\frac\pi 2\)  in \(\Omega\). In particular, $\varphi$ solves the Euler-Lagrange equation
  $\Delta (2 \varphi) + \kappa^2 \sin (2 \varphi) = 0$  in
  {\Om}. The radial symmetry of $\varphi$ then follows from {\cite{Gidas1979}}.
\end{proof}

In the case of the penalization of the boundary datum, we use a reflection
method introduced in {\cite{Lopes}} and the unique continuation principle for
elliptic equations (see, for instance, {\cite{Muller}}). Note that this method also works for the boundary value problem associated with $\mathcal{E}_{\kappa, 0}$, and the following proof also covers \Cref{corollaryDirichletRadial}.

\begin{proof}[Proof of \Cref{prop:symm}]
 We concentrate on the case \(\gamma>0\). Without loss of generality, one can assume that {\m} is not a constant minimizer. By
  \Cref{structureMinimizers}, there exists $\sigma \in
  \Othreethree$ and a solution $\varphi \in \mathcal C^\infty (\overline B_R)$ of \eqref{phasePenalization} such that $\m = \sigma (\sin \varphi, 0, \cos \varphi)$ and \(0<\varphi<\frac\pi 2\) in \(\overline B_R\). As before, we get that $\varphi$ is a solution of
  \begin{equation}
    \Delta (2 \varphi) + \kappa^2 \sin (2 \varphi) = 0 \quad \text{in } B_R .
    \label{equationElliptic}
  \end{equation}
  Now, let $H$ be a hyperplane passing through the origin and dividing $\RR^N$
  into two half-spaces $H^+$ and $H^-$. Up to interchange $H^+$ and $H^-$, one
  can assume that
  \begin{align*}
  \int_{H^- \cap \, B_R} \abs{\nabla \varphi}^2 \,+\, \kappa^2 \cos^2 \varphi 
     &\,+\, \frac{1}{\gamma^2}  \int_{H^- \cap \, \partial B_R} \sin^2 \varphi \; \\
     &\leqslant \int_{H^+ \cap \, B_R} \abs{\nabla \varphi}^2 \,+ \,\kappa^2 \cos^2
     \varphi \,+ \,\frac{1}{\gamma^2}  \int_{H^+ \cap \, \partial B_R} \sin^2
     \varphi .
  \end{align*}
  Let $\varphi_{\ast} \in H^1 \left( B_R \right)$ be defined by
  $\varphi_{\ast} = \varphi$ on $H^- \cap B_R$ and $\varphi_{\ast} = \varphi
  \circ \sigma_H$ on $H^+ \cap B_R$ where $\sigma_H$ stands for the reflection
  through $H.$ By the previous inequality, we have that $\Ekg (\varphi_{\ast})
  \leqslant \Ekg (\varphi)$, i.e., $\varphi_{\ast}$ is also a global
  minimizer. Hence it also solves {\eqref{equationElliptic}}. However, since
  $\varphi_{\ast} = \varphi$ on $H^- \cap B_R$, we deduce by the unique
  continuation principle (see Theorem III in {\cite{Muller}}) that
  $\varphi_{\ast} = \varphi$, i.e., $\varphi = \varphi \circ \sigma_H$ in \(B_R\). Since the hyperplane $H$ is arbitrary, this means that $\varphi$ is radially symmetric. This means that we can write \(\varphi(x)=\frac{u(\abs{x})}{2}\) for every \(x\in B_R\), for some function \(u:[0,R]\to [0,\pi]\). Moreover, since \(\varphi\) is smooth, \(u\) is smooth.
  
We now argue that \(u\) is non-increasing. Indeed, define the non-increasing  rearrangement of \(u\) by \(u^\ast(r)=\sup_{s\in [r,R]}u(s)\). We have that \(u^\ast\) is Lipschitz with \(\abs{(u^\ast)'}\leqslant\abs{u'}\) on \([0,R]\) since if \(0\leqslant r_1\leqslant r_2\leqslant R\), then \(u^\ast(r_2)\leqslant u^\ast(r_1)\) and
\[
u^\ast(r_1)\leqslant \sup_{s\in [r_2,R]} u(s)+\sup_{s\in [r_1,r_2]}\abs{u(s)-u(r_2)} \leqslant u^\ast(r_2)+(r_2-r_1)\sup_{s\in[r_1,r_2]}\abs{u'(s)}.
\]
But then, the function \(\varphi^\ast\in W^{1,2}(B_R)\), defined by \(\varphi^\ast(x)=u^\ast(\abs{x})\) for every \(x\in B_R\), satisfies \(\varphi=\varphi^\ast\) a.e.\  on \(\partial B_R\), and \(\cos \varphi^\ast\leqslant\cos \varphi\) and \(\abs{\nabla\varphi^\ast}\leqslant\abs{\nabla\varphi}\) a.e.\  in \(B_R\). Hence, \(\cos \varphi^\ast=\cos \varphi\) a.e., and so \(\varphi=\varphi^\ast\) a.e., since otherwise, \(\varphi^\ast\) would have strictly less energy than \(\varphi\) in \eqref{phasePenalization}.
    
Finally, as a solution of \eqref{phasePenalization}, \(\varphi(x)=\frac{u(\abs{x})}{2}\) must be a solution of the associated Euler-Lagrange equation, which means that \(u\) solves the system \eqref{equationELradial}-\eqref{boundaryConditionEL}.
\end{proof}

\section*{Acknowledgments}

{\noindent}G.~Di~F. acknowledges the support of the Austrian Science Fund
(FWF) through the special research program {\it{Taming complexity in partial
differential systems}} (grant F65). G.~Di~F. and V.~S. would also like to
thank the Max Planck Institute for Mathematics in the Sciences in Leipzig for
support and hospitality. A.~M. and V.~S. acknowledge support by Leverhulme grant
RPG-2018-438. All authors also acknowledge support from the Erwin
Schr{\"o}dinger International Institute for Mathematics and Physics (ESI) in
Vienna, given on the occasion of the workshop on New Trends in the
{\it{Variational Modeling and Simulation of Liquid Crystals}} held at ESI on
December 2--6, 2019.

\vskip 0.1cm

{\noindent}{\tmname{Giovanni}} {\tmname{Di Fratta}}, TU Wien, Institute of
Analysis and Scientific Computing, Wiedner Hauptstra{\ss}e 8-10,
1040 Wien, Austria.

{\tmem{E-mail address:}}
\href{mailto:giovanni.difratta@asc.tuwien.ac.at}{giovanni.difratta@asc.tuwien.ac.at}


{\noindent}{\tmname{Antonin}} {\tmname{Monteil}}, University of Bristol,
School of Mathematics, Fry building, Woodland Road, Bristol BS8 1UG, United
Kingdom.

{\tmem{E-mail address:}}
\href{mailto:antonin.monteil@bristol.ac.uk}{antonin.monteil@bristol.ac.uk}


{\noindent}{\tmname{Valeriy}} {\tmname{Slastikov}}, University of Bristol,
School of Mathematics, Fry building, Woodland Road, Bristol BS8 1UG, United
Kingdom.

{\tmem{E-mail address:}}
\href{mailto:Valeriy.Slastikov@bristol.ac.uk}{valeriy.slastikov@bristol.ac.uk}

\end{document}